\documentclass{amsart}
\usepackage[utf8]{inputenc}
\usepackage{geometry}

\usepackage{graphicx}
\usepackage{amsmath, amsfonts, amssymb,amsthm,mathtools,bm,color,mathrsfs}

\usepackage{latexsym}
\usepackage{amscd}
\usepackage[colorlinks,linkcolor=blue,anchorcolor=blue,citecolor=blue]{hyperref}
\usepackage{multirow}
\usepackage{indentfirst}
\usepackage{enumerate}
\usepackage[english]{babel}\usepackage{csquotes}
\usepackage{comment}
\usepackage{appendix}
\usepackage{url}

\newtheorem{theorem}{Theorem}[section] 
\newtheorem{lemma}[theorem]{Lemma}     

\newtheorem{proposition}[theorem]{Proposition}
\newtheorem{conjecture}[theorem]{Conjecture}
\newtheorem*{problem*}{Problem}

\theoremstyle{definition}
\newtheorem{definition}[theorem]{Definition}

\newtheorem*{remark}{Remark}

\newcommand{\CC}{\mathbb C}

\newcommand{\PP}{\mathbb P}
\newcommand{\GG}{\mathbb G}
\newcommand{\ord}{\operatorname{ord}}
\newcommand{\exc}{\operatorname{exc}}

\title[A truncated second main theorem for algebraic tori with moving targets]{A truncated  second main theorem for algebraic tori \\
with moving targets and applications}
\author{Ji Guo}
\address{Institute of Mathematics\\ Academia Sinica\\6F, Astronomy-Mathematics Building\\No. 1, Sec. 4, Roosevelt Road \\ Taipei
10617\\Taiwan} \email{jiguo@gate.sinica.edu.tw}

\author{Chia-Liang Sun}
\address{Institute of Mathematics\\ Academia Sinica\\6F, Astronomy-Mathematics Building\\No. 1, Sec. 4, Roosevelt Road \\ Taipei
10617\\Taiwan} \email{ csun@math.sinica.edu.tw}

\author{Julie Tzu-Yueh Wang}
\address{Institute of Mathematics\\ Academia Sinica\\6F, Astronomy-Mathematics Building\\No. 1, Sec. 4, Roosevelt Road \\ Taipei
10617\\Taiwan} \email{jwang@math.sinica.edu.tw}

\begin{document}


\begin{abstract}
We establish a second main theorem for algebraic tori with slow growth moving targets with truncation to level 1.  As the first application of this result, we prove the
 Green-Griffith-Lang conjecture for projective spaces with $n+1$ components in the context of moving targets of slow growth.  
 Then we discuss the integrability of the ring of exponential polynomials in the ring of entire functions as another application.
 \end{abstract}

\thanks{2010\ {\it Mathematics Subject Classification.} Primary 30D35,  
Secondary 30A70,11J97.}
\thanks{The second named was supported by  Taiwan's MoST grant 109-2811-M-001-613  and the third named author was supported in part by Taiwan's MoST grant 108-2115-M-001-001-MY2.}

\baselineskip=16truept 
\maketitle \pagestyle{myheadings}

\section{Introduction}
Originating in the work of Osgood, Vojta, and Lang, it has been observed that there is a striking correspondence between many statements in Nevanlinna theory and statements in Diophantine approximation.  A detailed ``dictionary" between the two subjects has been constructed by Vojta \cite{Vo}.  The following conjecture can be viewed as   the complex analogue of Vojta's generalized abc conjecture (\cite[Conjecture 23.4]{vojta2009diophantine}).
\begin{conjecture}\label{ConjABC}
Let $X$ be a smooth complex projective variety, let D be a normal crossing divisor on $X$, let $K$ be the canonical divisor on $X$, and let $A$ be an ample divisor on $X$.  Then:
\begin{enumerate}
\item[{\rm (a)}] If $f:\mathbb C\to X$ be an algebraically nondegenerate analytic map, then
\begin{align}\label{truncate1}
N_f^{(1)}(D,r)\ge_{\exc} T_{K+D,f}(r)-{\rm o}(T_{ A,f}(r)).
\end{align}
\item[{\rm (b)}] For any $\epsilon>0$, there is a proper Zariski-closed subset $Z$ of $X$, depending only on $X$, $D$, $ A$, and $\epsilon$ such that  for any analytic map $f:\mathbb C\to X$ whose image is not contain in $Z$, the following
\begin{align}\label{truncate2}
N_f^{(1)}(D,r)\ge_{\exc} T_{K+D,f}(r)-\epsilon T_{ A,f}(r)
\end{align}
 holds.
\end{enumerate}
\end{conjecture}
Here, for positive integer $n$, $N_f^{(n)}(D,r)$ is the $n$-truncated counting function with respect to $D$ is given by
\begin{align}
N_f^{(n)}(D,r)=\sum_{0<|z|<r}\min\{{\rm ord}_z f^*D, n\}\log\frac{r}{|z|}+\min\{{\rm ord}_0 f^*D, n\}\log r,
\end{align}
$T_{D,f}(r)$ is the (Nenvanlinna) height function relative to the divisor $ D$ (See \cite[Section 12]{vojta2009diophantine}.), and the notion $\leq_{\exc}$  means that the estimate holds for all $r$ outside a set of finite Lebesgue measure.
Unlike the situation in number fields, there are some known results for this conjecture.  For example, the conjecture holds for $\dim X=1$.  The Cartan's second main theorem, where $X=\mathbb P^n$ and $D=H_1+\cdots+H_q$, where the $H_i$ are hyperplanes in general position, suggests that \eqref{truncate1} holds with 
$N_f^{(1)}(D,r)$ replaced by 
 $\sum_{i=1}^qN_f^{(n)}(H_i,r)$  under a weaker assumption that the map $f$ is linearly non-degenerate. When $X$ is a semiabelian variety, Noguchi, Winkleman and Yamanoi in \cite{noguchi2008semiabelian} showed that the inequality \eqref{truncate1} holds with $N_f^{(1)}(D,r)$ replaced by 
 $N_f^{(k_0)}(D,r)$ for some positive integer $k_0$, and \eqref{truncate2} holds if the map is algebraically nondegenerate.
 
The above conjecture is much harder for the case of moving targets, i.e. the divisor $D$ is defined over a field of ``small functions" with respect to the map $f$.  
The only existing result   in the moving case is due to Yamanoi in \cite{Yamanoi2004} for   $\dim X=1$.  All the other results in this direction are stated with a very high truncated level.(See Theorem \ref{SMTmoving}, which is a result of Dethloff and Tan, for an example.)   Our first result is to establish the inequality \eqref{truncate2} for complex tori with moving targets.
We recall that the small field with respect to  a holomorphic curve   $\mathbf{f}:\CC\to\PP^n(\CC)$  is given by
\begin{equation}\label{smallfield}
    K_{\mathbf{f}}:=\{a:a \text{ is a meromorphic function with } T_{a}(r)= {\rm o}( T_{\mathbf{f}}(r))\}.
\end{equation} 

\begin{theorem}\label{main_thm_1}
	Let $u_0,u_1,\dots,u_n$ be  nonconstant entire functions without zeros, i.e. $\mathbf{u}=(u_{0},\ldots,u_{n}):\mathbb C\to \mathbb G_m^{n+1}$. Let $K_{\mathbf{u}}$ be the small field with respect to $\mathbf{u}$. Let $G$ be a nonconstant  homogeneous polynomials in $K_{\mathbf{u}}[x_0,\hdots,x_n]$ with no repeated  nonmonomial factors in $K_{\mathbf{u}}[x_0,\hdots,x_n]$. If $u_0,\dots,u_n$ are multiplicatively independent modulo $K_{\mathbf{u}}$, then for any $\epsilon >0$	
\begin{align}\label{truncate_thm_1}
 N_{G(\mathbf{u})}(0,r)-N^{(1)}_{G(\mathbf{u})}(0,r)\le_{\exc} \epsilon T_{\mathbf{u}}(r).
\end{align}
	
If we assume furthermore that  none of  the functions $G(1,0,\dots, 0),\dots,$ $G(0,\dots,0,1)$  is identically zero, then
	\begin{align}\label{truncate_thm_2}
  N^{(1)}_{G(\mathbf{u})}(0,r)\ge_{\rm exc}  (\deg  G-\epsilon)\cdot T_{\mathbf{u}}(r). 
  	\end{align}
Moreover, if each $u_j$, $0\le j\le n$, is of finite order and $G_i\in \mathbb C[z][x_0,\hdots,x_n]$, $1\le i\le q$, then  the both assertions above hold  under the weaker assumption that
$u_1,\hdots,u_n$ are multiplicatively independent modulo $\mathbb C$.
\end{theorem}
Here,  $T_{\mathbf{u}}(r)$ is the Nevanlinna height function associated to ${\mathbf{u}}$ and $N_{G(\mathbf{u})}(0,r)$ ($N^{(1)}_{G(\mathbf{u})}(0,r)$ respectively) is the counting function associated to 0 and $G(\mathbf{u})$ (with truncation to level $1$  respectively) to be defined in the next session.  
The first application of Theorem \ref{main_thm_1} concerns  a case of the
 Green-Griffith-Lang conjecture for projective spaces in the context of moving targets of slow growth. 
 The Green-Griffiths-Lang conjecture in the non-compact case (see \cite[Proposition~15.3]{vojta2009diophantine}) 
 reads as follows: 
 {\it If $X$ is  a smooth variety, $D$ is a normal crossing divisor on $X$, and $X\setminus D$ is a variety of log general type, then a holomorphic map $f:\mathbb C\to X\setminus D$ cannot have Zariski-dense image.}  When $X=\mathbb P^n$, the condition for $X\setminus D$ to be of log general type is equivalent to the inequality $\deg D\ge n+2$.   In this setting, the conjecture is verified by Green  \cite{green1975some} when $D$ has at least $n+2$ components, and when $n=2$ of 3 components with $\deg D=4$ in \cite{green1974functional} under the assumption that $f$ is of finite order.
 The $n+1$ component case (with $\deg D\ge n+2$) is solved by  Noguchi, Winkelman and Yamanoi  in \cite[Theorem 5.4]{noguchi2007degeneracy}.  For the moving target situation, the $n+2$ component case follows directly from the second main theorem for moving hypersurfaces in \cite{DethloffTan2011}; the boundary case of $n+1$ components, i.e. $\deg D=n+2$, is recently established in \cite{GSW20}.  The following theorem gives a full treatment to the case of $n+1$ components.

\begin{theorem}\label{GG_conj}
Let $\mathbf{f}=(f_0,\hdots,f_n):\CC\to \PP^n$ be a holomorphic map, where $f_0,\hdots,f_n$ are entire functions without common zeros.    Let 
$F_i$, $1\le i\le n+1$ be homogeneous irreducible polynomials of positive degree  in $K_{\mathbf{f}}[x_0,\hdots,x_n]$ such that $\sum_{i=1}^{n+1}\deg F_i\ge n+2$. 
Assume that there exists $z_0\in \CC$ such that all the coefficients of  all $F_i$, $1\le i\le n+1$ are  holomorphic at $z_0$  and the zero locus of  $F_i$,  $1\le i\le n+1$, evaluating at $z_0$  intersect transversally.  If   $F_i(\mathbf{f})$, $1\le i\le n+1$, are entire function without zeros, then $\mathbf{f}$ is algebraically degenerate over $K_\mathbf{f}$. 
 Moreover, if $\mathbf{f}$ is of finite order and $F_i\in \mathbb C [x_0,\hdots,x_n]$ for each $i$, then $\mathbf{f}$ is algebraically degenerate over $\mathbb C$. 
\end{theorem}

The second application of Theorem \ref{main_thm_1} is related to an algebraic problem of exponential polynomials.
We recall  the definition of  exponential polynomials of order $q\in\mathbb N$ as follows.
\begin{definition}
An  {\it   exponential polynomial of order $q$}  is an entire function  of the form 
\begin{align}\label{exppoly}
f(z)=P_1(z)e^{Q_1(z)}+\cdots+P_m(z)e^{Q_m(z)},
\end{align} 
where $m\in\mathbb N$, $P_i,Q_i\in\CC[z]$, $1\le i\le m$, such that  
$$
\max_{1\le i\le m}\{\deg Q_i\}=q.
$$
\end{definition}
Our interest of studying exponential polynomials comes from its correspondence with linear recurrences and its algebraic structure. 
In  \cite{GSW20}, we showed that if an exponential polynomial $f$ is a $d$-th power of some entire function $g$, i.e. $f=g^d$, then $g$ is also an exponential polynomial.
In view of the correspondence between Diophatine approximation  and Nevanlinna theory, this result is a
  complex analogue of  Pisot's $d$-th root conjecture for linear recurrences, which is  completely solved in  \cite{zannier2000proof}.
From the view points of studying the algebraic structure of the exponential polynomials,  it basically says that the radical of an exponential polynomial is also an exponential polynomial.  We refer to  \cite{heittokangas2018zero},  \cite{guo2019quotient}, and \cite{guo2019quotient2}  for other expositions in this direction; and \cite{CGPastenLogic} for related results in logic.  The following is a  nature question  on the  integrality of the ring of exponential polynomials over the ring of entire functions.

\begin{problem*} 
Let $q$ and $d$   be  positive integers.  Denote by $\mathcal K_q$ the quotient field of the ring $\mathcal E_{q}$ of the exponential polynomials of order at most $q$.
Let  $A_0,\hdots,A_{d-1} \in\mathcal E_{q} $ such that    the polynomial $F(Y):=Y^d+A_{d-1}Y^{d-1}+\hdots+A_1Y+A_0$ is irreducible over $\mathcal K_q[Y]$.  Suppose that there exists an entire function $g$ such that 
$F(g)=0$. Is it true that $g\in \mathcal E_{q}$? 
\end{problem*}

For the case $q=1$ and each $A_i$ is an exponential polynomial with constant coefficients, refer to \cite{ritt1929algebraic} for a complete solution. 

Being not able to solve the problem completely, we prove the following generic result. 
  
\begin{theorem}\label{mainthm3}
Let $d$ and $q_i$, $1\le i\le n$ be  positive integers and let $q=\max\{q_1,\hdots,q_n\}$.  Let $\mu_1,\hdots,\mu_n\in\mathbb C^*$  and
$\mathbf u=(u_1,\hdots,u_n)=(e^{\mu_1z^{q_1}}, \hdots, e^{\mu_nz^{q_n}})$.  Assume that $u_1,\hdots,u_n$ are multiplicatively independent modulo $\mathbb C$. Let $A_{i}\in  \mathbb{C}(z)[u_{1}^{\pm 1},\hdots,u_{n}^{\pm 1}] $,
$0\le i\le d-1$. Assume that the polynomial  $F(Y):=Y^{d}+A_{d-1}Y^{d-1}+\hdots+A_{1}Y+A_{0}\in \mathcal K_q[Y]$ 
is irreducible  and   its discriminant $\Delta$ is square-free   in  $\mathbb{C}(z)[u_{1}^{\pm 1},\hdots,u_{n}^{\pm 1}]$,  i.e. $\Delta$  has no non-unit  repeated factor.
Suppose that  there exists an entire function $g$ such that 
$F(g)=0$. Then $g\in \mathcal E_{q}$.  
\end{theorem}
Let us explain how this theorem is related to the problem above.  Let $A_0,\hdots,A_{d-1}\in \mathcal E_{q}$ such that at least one of the $A_i$ is of order $q$.  Then we can find   units of finite order $u_1,\hdots,u_n$, which are multiplicative independent modulo $\mathbb C$,  equivalently algebraic independent over $\CC(z)$, as in Theorem \ref{mainthm3}  such that each $A_i\in \mathbb C[z][u_{1}^{\pm 1},\hdots,u_{n}^{\pm 1}]\subset \mathbb C(z)[u_{1}^{\pm 1},\hdots,u_{n}^{\pm 1}]$. Note that the latter ring is a unique factorization domain.  We refer to the proof of Theorem 1.3 in \cite{GSW20} for such constructions. 
Theorem \ref{mainthm3} then applies to  the problem under this identification.

Our proof of Theorem \ref{main_thm_1} relies on the GCD theorem \cite{levin2019greatest} of Levin and the third author  and the machinery developed in  \cite{GSW20}.   The proof of  Theorem \ref{GG_conj}  follows the ideas in \cite{noguchi2007degeneracy} and \cite{corvaja2013algebraic} with extension to the moving situation. We note that Capuano and Turchet in \cite{CT} generalized the work of \cite{corvaja2013algebraic} to non-split function fields (i.e. the moving case in our terminology) for surfaces.  Finally, the proof of Theorem \ref{mainthm3} is an adaption and generalization of Theorem 3 in \cite{corvaja2013algebraic}  to the complex situation.

Some background materials will be given in the next session.  The key lemmas are stated in Session \ref{main_lem}.
The proofs of our theorems will be given in Section  \ref{thm1}-\ref{thm3} respectively.
   
  \section{Preliminaries}\label{preliminaries}
We will give relevant materials and derive some basic results in this session.

\subsection{Nevanlinna Theory}
We will set up some notation and definitions  in Nevanlinna theory and recall some basic results.   We refer to \cite{lang1987introduction} and \cite{ru2001nevanlinna} for details.

Let $f$ be a meromorphic function  and   $z\in \CC$ be a complex number. Denote $v_z(f):=\ord_z(f)$,
$$v_z^+ (f):=\max\{0,v_z(f)\}, \quad\text{and }\quad  v_z^- (f):=-\min\{0,v_z(f)\}.$$
Let $n_f(\infty,r)$ (respectively,  ${n}^{(Q)}_{f}(\infty,r)$) denote the number of poles of $f$ in $\{z:|z|\le r\}$, counting multiplicity (respectively, ignoring multiplicity larger than $Q\in\mathbb N$). The  {\it counting function} and {\it truncated counting function} of $f$ of order $Q$ at $\infty$  are  defined respectively by
\begin{align*}
N_f(\infty,r)&:=\int_0^r\frac{n_f(\infty,t)-n_f(\infty,0)}t dt+n_f(\infty,0)\log r\\
&=\sum_{0<|z|\le r } v_z^- (f)\log |\frac{r}{z}|+v_0^- (f)\log r,
\end{align*}
and
\begin{align*}
N^{(Q)}_f(\infty,r)&:=\int_0^r\frac{n^{(Q)}_f(\infty,t)-n^{(Q)}_f(\infty,0)}t dt+n^{(Q)}_f(\infty,0)\log r\\
&=\sum_{0<|z|\le r } \min\{Q,v_z^- (f)\}\log |\frac{r}{z}|+\min\{Q,v_0^- (f)\}\log r.
\end{align*} 
Then define the {\it counting function} $N_f(r,a)$ and the {\it truncated counting function} $N^{(Q)}_f(r,a)$ for $a\in\CC$ as
$$
N_f(a,r):=N_{1/(f-a)}(r, \infty)\quad\text{and}\quad N^{(Q)}_f(a,r):=N^{(Q)}_{1/(f-a)}(\infty,r).
$$
The  {\it proximity function} $m_f(\infty,r)$ is defined by
$$
m_f(\infty,r):=\int_0^{2\pi}\log^+|f(re^{i\theta})|\frac{d\theta}{2\pi},
$$
where $\log^+x=\max\{0,\log x\}$ for  $x\ge 0$. For any $a\in \CC,$ the {\it proximity function} $m_f(a,r)$ is defined by
$$m_f(a,r):=
m_{1/(f-a)}(\infty,r).
$$
The {\it characteristic function} is defined by
$$
T_f(r):=m_f(\infty,r)+N_f(\infty,r).
$$
It  satisfies the First Main Theorem as follows.
\begin{theorem}\label{Cfirstmain}  
  Let $f$ be a non-constant  meromorphic function on $\CC$.  Then for every $a\in\CC$ and for any positive real number $r$, $$m_f(a,r)+N_f(a,r)=T_f(r)+O(1).$$ 
  where $O(1)$ is independent of $r$.
\end{theorem}

Rational functions over $\mathbb C$ can be characterized by characteristic functions as follows.  (See\cite[Chapter VI, Theorem 2.6]{lang1987introduction}.)
\begin{theorem}\label{rational}  
Let $f$ be a meromorphic function on $\mathbb C$.
Then $f$ is a rational function, i.e. $f\in\mathbb C(z)$, if and only if $T_f(r)=O(\log r)$ for $r\to\infty$.
\end{theorem}
We recall the lemma on the logarithmic derivative. (See \cite[Theorem~A1.2.5]{ru2001nevanlinna}.)
\begin{lemma}
 \label{lemma_log_deri} Let $f$ be a non-constant meromorphic function on $\CC$. Then for any $\varepsilon>0$, 
   \begin{equation*}
       m_{f'/f}(\infty,r)\le_{\exc} \log T_f(r)+(1+\varepsilon)\log^+\log T_f(r)+O(1).
   \end{equation*}
\end{lemma}

Let ${\mathbf f}: \CC \rightarrow \PP^n(\CC)$ be a  holomorphic map and $ (f_0,\dots,f_n)$ be a reduced representation of  ${\mathbf f}$, i.e. $f_0,\dots, f_n$ are  entire functions on $\CC$ without common zeros. The Nevanlinna-Cartan {\it characteristic function}  $T_{\mathbf f}(r)$ is defined by
$$T_{\mathbf f}(r) =   \int_0^{2\pi} \log\max\{|f_0(re^{i\theta})|,\dots ,|f_n(re^{i\theta})|\} \frac{d\theta}{2\pi}+O(1).$$
This definition is independent, up to an additive constant, of the choice of the reduced representation of ${\mathbf f}$. 
 
 We will make use of the following   elementary inequality . (See \cite[Lemma 2.5]{levin2019greatest} for a proof.)

 \begin{lemma}
 \label{gineq}
 Let $g_1,\ldots, g_n$ be meromorphic functions.  Let ${\bf g}:=(1,g_1,\dots,g_n):\CC\to\PP^n.$  Then 
 \begin{align*}
 T_{g_i}(r)\le T_{\bf g}(r)\leq \sum_{i=1}^nT_{g_i}(r)+O(1),
 \end{align*}
 for $1\le i\le n$.
 \end{lemma}

We will use the following version of a truncated second main theorem.

 \begin{theorem}[{\cite[Theorem 2.1]{ru2004truncated}}]\label{trunborel}  
   Let $\mathbf{f}=(f_0,\hdots,f_n):\CC\to\PP^n(\CC)$ be a holomorphic map with $f_0,\dots,f_n$ entire and no common zeros. Assume that $f_{n+1}$ is a holomorphic function satisfying the equation $f_0+\dots+f_n+f_{n+1}=0$. If $\sum_{i\in I}f_i\ne 0$ for any proper subset $I\subset\{0,\dots,n+1\}$, then 
    \begin{equation*}
        T_{\mathbf{f}}(r)\leq_{\exc} \sum_{i=0}^{n+1} N_{f_i}^{(n)}(0,r)+O(\log^+T_{\mathbf{f}}(r)).
    \end{equation*}
\end{theorem}

We now recall the following second main theorem for hypersurfaces with moving targets from  \cite{DethloffTan2011}.  
Let ${\mathbf f}: \CC \rightarrow \PP^n(\CC)$ be a  holomorphic map. We denote by  
\begin{equation*}
    K_{\mathbf{f}}:=\{a:a \text{ is a meromorphic function with } T_{a}(r)= o(T_{\mathbf f}(r)) \}.
\end{equation*} 

A set $\{Q_1,\dots,Q_q\}$ of homogeneous polynomials in $K_{\mathbf{f}}[x_0,\dots,x_n]$ is said to be in weakly general position if there exists $z_0\in\CC$ in which all coefficient functions of all $Q_j$, $j=1,\dots,q$ are holomorphic and such that for any $1\le j_0<\cdots<j_n\le q$ the system of equations 
$\left\{
Q_{j_t}(z_0)(x_0,\dots,x_n)=0: 0\le t\le n
\right\}$
has only the trivial solution $(x_0,\dots,x_n)=(0,\dots,0)\in\CC^{n+1}$.

The following statement is a direct consequence of  the main theorem in \cite{DethloffTan2011}.

\begin{theorem}[\cite{DethloffTan2011}]\label{SMTmoving}
Let $\mathbf{f}$ be a nonconstant meromorphic map of $\CC$ into $\PP^n$.   Let $K\subset K_{\mathbf{f}}$ be a subfield. Let $\{Q_j\}_{j=1}^q$ be a set of homogeneous polynomials in $K[x_0,\dots,x_n]$  in weakly general position and with $\deg Q_j=d_j\ge 1$. Assume that $\mathbf{f}$ is algebraically non-degenerate over $K$. Then for any $\epsilon>0$, there exists a positive integer $L$ depending only on $n$, $\epsilon$ and $d_j$, $1\le j\le q$, such that the following inequality holds:
\begin{equation*} 
(q-n-1-\epsilon)T_{\mathbf{f}}(r)\le_{\exc} \sum_{j=1}^q\frac{1}{d_j} N_{Q_j(\mathbf{f})}^{(L)}(0,r).
\end{equation*}
\end{theorem}
We note that the integer $L$ in this theorem is large as mentioned in \cite[Proposition 1.2]{DethloffTan2011}.

We will need the following version of Hilbert Nullstellensatz reformulated from \cite[Proposition 2.1]{DethloffTan2011}. (See also \cite[Chapter~XI]{waerden1967}.)
\begin{proposition}[{\cite[Proposition 2.1]{DethloffTan2011}}]\label{HilbertN}
Let $\mathbf{f}$ be a nonconstant meromorphic map of $\CC$ into $\PP^n$.  Let $K\subset K_{\mathbf{f}}$ be a subfield.  Let $\{Q_i\}_{i=1}^{n+1}$ be a set of homogeneous polynomials in $ K [x_0,\dots,x_n]$  in weakly general position and with $\deg Q_j=d_j\ge 1$.   Then there exists a positive integer $s$, $R\in K $ not identically zero and  $P_{ji}\in K[x_0,\dots,x_n]$, $1\le i,j\le n+1$,  such that 
$$
x_j^s\cdot R=\sum_{i=1}^{n+1} P_{ji} Q_i
$$
for each $0\le j\le n.$
\end{proposition}

Finally, we recall  the following definitions.
 A meromorphic function $f$ is {\it of finite order} (or more exactly, {\it of order $q$}) if 
     \begin{equation*}
         \limsup_{r\to \infty}\frac{\log T_f(r)}{\log r}=q.
    \end{equation*}
    In addition, if $f$ is a unit, i.e. an entire function  without zeros,  of order $q$, then  $u=e^P$ for some  $P\in\CC[z]$ of degree $q$.
Similarly, a map $\mathbf{f}:\CC\to \PP^n(\CC)$ is   said to be  {\it of finite order} (or more exactly, {\it of order $q$}) if 
     \begin{equation*}
       \limsup_{r\to \infty}\frac{\log T_{\mathbf{f}}(r)}{\log r}=q.
    \end{equation*}
\subsection{GCD in Nevanlinna theory}\label{gcd}
We  recall the gcd counting function of two meromorphic functions and a gcd theorem with moving targets from   \cite{levin2019greatest}.
 
Let $f$ and $g$ be meromorphic functions.  We let
\begin{equation*}
    n(f,g,r):=\sum_{|z|\le r}\min\{v_z^+(f),v_z^+(g)\}
\end{equation*}
and 
\begin{equation*}
    N_{\gcd}(f,g,r):=\int_0^r \frac{n(f,g,t)-n(f,g,0)}{t}dt+ n(f,g,0)\log r.
\end{equation*}
The gcd theorem we need below is a little more general than what is stated in \cite[Theorem~5.1]{levin2019greatest}, and it follows clearly from their proof. 
 
\begin{theorem}[{\cite[Theorem~5.1]{levin2019greatest}}]\label{gcd_moving_unit}
    Let $u_1,\dots,u_n$ be entire functions without zeros.   Let $K\subset K_{\mathbf{u}}$ be a subfield, and let $F,G\in K[x_1,\dots,x_n]$ be nonconstant coprime polynomials. Then  for every $\varepsilon>0$  the following inequality holds
    \begin{equation*}
        N_{\gcd}(F(u_1,\dots,u_n),G(u_1,\dots,u_n),r)\le_{\exc}\varepsilon \max_{1\le j\le n}\{T_{u_j}(r)\},    \end{equation*}
        if $u_1,\dots, u_n$ are algebraically independent over $K$.
     \end{theorem}

 \begin{remark}
The proof of \cite[Theorem~5.1]{levin2019greatest} actually only treat the case where   $n\ge2$ because the case where $n=1$ is much easier, as shown in the following argument below.  Let $u$ be an entire function without zero.
Suppose that $F$  and $G$ are coprime polynomials in $K[x ]$.
 Then there exist $A,B\in K [x]$ such that $AF+BG=1$.  Then  $A(u)F(u)+B(u)G(u)=1$ and hence
\begin{align}\label{poles}
\min\{v_z^+(F(u)),v_z^+(G(u))\}\le \max\{v_z^-(A(u)),v_z^-(B(u))\}.
\end{align}
We note that $F(u)$ and $G(u)$ are not indentical zero under the assumption that $u$ is not algebraic over $K$.
 Since $u$ has no pole, the right hand side of \eqref{poles} is bounded by the number of poles of the coefficients of $A$ and $B$.   By  Theorem \ref{Cfirstmain},  $N_{\beta}(\infty,r)\le T_{\beta}(r)+O(1)={\rm o}(  T_{u}(r))$ for any $\beta\in K$.  Therefore, \eqref{poles} implies that  
\begin{align*}
            N_{\gcd}(F(u),G(u),r)\le {\rm o}(  T_{u}(r)).
 \end{align*}
\end{remark}

 \subsection{Basic Results}
\begin{lemma}\label{a'}
Let ${\mathbf f}: \CC \rightarrow \PP^n(\CC)$ be a  holomorphic map. Then $a'\in K_{\mathbf{f}}$ for every $a\in K_{\mathbf{f}}$.
\end{lemma}
 \begin{proof}
 We may suppose that $a$ is non-constant. Noting that $N_{a'/a}(\infty,r)\le N^{(1)}_a(0,r)+ N^{(1)}_a(\infty,r)$, Theorem \ref{Cfirstmain}  implies  $N_{a'/a}(\infty,r)\le 2T_a(r)+O(1)$. On the other hand, we have from Lemma  \ref{lemma_log_deri} that  $m_{a'/a}(\infty,r)\le_{\exc} O(\log T_a(r))$. Now  the condition $a\in K_{\mathbf{f}}$ implies that $T_{a'/a}(r)=  o(T_{\mathbf f}(r))$ and thus $T_{a'}(r)\le T_{a'/a}(r)+T_a(r)=  o(T_{\mathbf f}(r))$ as desired. 
 \end{proof}

We refer to \cite{GSW20} for proofs of the following statements.
\begin{lemma}[{\cite[Lemma 2.6]{GSW20}}]\label{u'/u}
Let $f$ be a nonconstant meromorphic function satisfying 
\begin{align*} 
N^{(1)}_f(0,r)+N^{(1)}_f(\infty,r)={\rm o}( T_f(r)).
\end{align*}
Then $T_{f'/f}(r)\le_{\exc} {\rm o}( T_f(r))$. 
\end{lemma}
 \begin{proposition}[{\cite[Corollary 2.8]{GSW20}}]\label{Borel1co}
Let   $u_0,\dots,u_n$ be  units, i.e. entire functions without zeros, and let ${\bf u}=(u_0,\hdots,u_n)$.  Let $K$ be a subfield of $K_{\bf u}$.
If $u_0,\dots,u_n$ are algebraically dependent over $K$, then they are multiplicatively dependent modulo $ K_{\bf u}$.
\end{proposition}
\begin{proposition}[{\cite[Corollary 2.10]{GSW20}}]\label{Borel3}
Let $n\ge 2$ and $Q_1,\dots,Q_n\in \CC[z].$  If $e^{Q_1}, \dots,e^{Q_n}$ are algebraically dependent over $\mathbb C(z)$, then they are multiplicatively dependent modulo $\mathbb C$.
\end{proposition}

\section{Main Lemmas}\label{main_lem}
Let $u_1,...,u_n$ be nonconstant units, i.e. entire functions without zeros.  
For convenience of discussions in the affine situation,  we simply denote by  $\mathbf{u}=(u_{1},\ldots,u_{n})$ and 
\begin{equation}\label{affinesmallfield}
    K_{\mathbf{u}}:=\{a:a \text{ is a meromorphic function with } T_{a}(r)= o( \max_{1\le i\le n}T_{u_i}(r)) \}.
\end{equation} 
Let  $\mathbf{x}:=(x_{1},\ldots,x_{n})$.  For $\mathbf{i}=(i_{1},\ldots,i_{n})\in\mathbb{Z}^{n}$, we let
$\mathbf{x^{i}}:=x_{1}^{i_{1}}\cdots x_{n}^{i_{n}}$ and  $\mathbf{u^{i}}:=u_{1}^{i_1}\cdots u_{n}^{i_{n}}$.
For a polynomial $F(\mathbf{x})=\sum_{\mathbf{i}}a_{\mathbf{i}}\mathbf{x}^{\mathbf{i}}\in K_\mathbf{u}[\mathbf{x}]:=K_\mathbf{u}[x_1,\dots,x_n]$,
we define 
\begin{align}\label{Duexpression}
D_{\mathbf{u}}(F)(\mathbf{x})
:=\sum_{\mathbf{i}}\frac{(a_{\mathbf{i}}\mathbf{u}^{\mathbf{i}})'}{\mathbf{u}^{\mathbf{i}}}\mathbf{x}^{\mathbf{i}}
=\sum_{\mathbf{i}}(a'_{\mathbf{i}}+a_{\mathbf{i}}  \cdot\sum_{j=1}^{n}i_j\frac{u_{j}'}{u_{j}})\mathbf{x}^{\mathbf{i}}.
\end{align} 
We note that $D_{\mathbf{u}}(F)(\mathbf{x})\in K_\mathbf{u}[\mathbf{x}]$  since $u_i'/u_i\in K_\mathbf{u}$, $1\le i\le n$, by Lemma \ref{u'/u}.   
Furthermore,  a direct computation shows that 
\begin{align}\label{fuvalue}
F(\mathbf{u})'=D_{\mathbf{u}}(F)(\mathbf{u}),
\end{align}
and   the following product rule: 
\begin{align}\label{productrule}
D_{\mathbf{u}}(FG)=D_{\mathbf{u}}(F)G+FD_{\mathbf{u}}(G)
\end{align}
for  $F,G\in K_\mathbf{u}[\mathbf{x}]$.
We recall the following  from  \cite{GSW20}.
\begin{lemma}[{\cite[Lemma 3.1]{GSW20}}]\label{coprime}
Let   $u_1,...,u_n$ be nonconstant  entire functions without zeros and $\mathbf{u}=(u_{1},\ldots,u_{n})$.
Let $K$ be a subfield of  $K_\mathbf{u}$ such that $u_j'/u_j\in K$, $1\le j\le n$, and $a'\in K$ for any $a\in K$.
Let $F$ be a nonconstant polynomial in $K[x_1,\dots,x_n]$ without  monomial factors and let
$F=F_1^{d_1}\cdots F_k^{d_k}$, where $k\ge 1$ and  $F_1,\hdots,F_k \in  K[x_1,\dots,x_n] $ are distinct irreducible  factors of $F$.  Suppose that $u_1,\dots,u_n$ are  multiplicatively independent modulo $K$.  Then the following two polynomials $\bar{F}:=F_1\cdots F_k$, 
and $\hat F_{\mathbf{u}}:=d_1D_{\mathbf{u}}(F_1)F_2\cdots F_k+\cdots +d_kF_1\cdots F_{k-1}D_{\mathbf{u}}(F_k)$
 are coprime in $K[x_1,\dots,x_n]$. 
\end{lemma}

The following lemma is a reformulation of \cite[Lemma 3.2]{GSW20} with $\alpha$ being an entire function.  
\begin{lemma}\label{dth_power_count}
Let   $u_1,...,u_n$ be nonconstant  entire functions without zeros and $\mathbf{u}=(u_{1},\ldots,u_{n})$.
Let $K$ be a subfield of  $K_\mathbf{u}$ such that $u_j'/u_j\in K$, $1\le j\le n$, and $a'\in K$ for any $a\in K$.
Let $d\ge 2$ be an integer. 
  Let $F\in K[x_1,\dots,x_n]$ and assume that $F$ has no nonmonomial repeated factors. 
Assume that $u_1,\dots,u_n$ are algebraically independent over $K$.  Let $\varepsilon >0$.
If  $ F(\mathbf{u})=\alpha g^d$ for some nonzero entire functions $g$ and $\alpha$, then 
\begin{equation*}
    N_{g}(0,r)\le_{\exc} \varepsilon \max_{1\le j\le n} \{T_{u_j}(r)\}.
\end{equation*}
\end{lemma}
 \begin{proof}
 Suppose that $ F(\mathbf{u})=\alpha g^d$ for some nonzero entire functions $g$ and $\alpha$. If $F=c\in K$, then $ N_{g}(0,r)\le N_{c}(0,r)=o( \max_{1\le j\le n} T_{u_j}(r))$ and thus the desired conclusion holds. Hence, we may assume $F\in K[x_1,\dots,x_n]$ is nonconstant.  
Since  $u_1,...,u_n$ are entire functions without zeros, we may remove the monomial factors from $F$ if necessary without changing our assertion.  Therefore we may further assume that $F$ has no monomial factors.
 
Let $F=F_1\cdots F_k$, where $F_1,\hdots,F_k \in  K[x_1,\dots,x_n] $ are distinct irreducible  factors of $F$.  

By \eqref{fuvalue}, we have  
\begin{align}\label{expression1}
 g^{d-1}(d\alpha g'+\alpha' g) =D_{\mathbf{u}}(F)(\mathbf{u}).
\end{align}
 Since $F$ and $D_{\mathbf{u}}(F)$ are coprime in $K[x_1,\dots,x_n]$ by
 Lemma \ref{coprime}, for a given $\varepsilon'>0$, Theorem \ref{gcd_moving_unit} gives 
    \begin{equation*}
        \begin{split}
            N_{\gcd}(F(\mathbf{u}),D_{\mathbf{u}}(F)(\mathbf{u}),r)\le_{\exc} \varepsilon' \max_{1\le j\le n} \{T_{u_j}(r)\}.
        \end{split}
    \end{equation*}
 
    On the other hand, \eqref{expression1} implies that
    \begin{equation}\label{gcdd} 
            N_{\gcd}(F(\mathbf{u}),D_{\mathbf{u}}(F)(\mathbf{u}),r)=N_{\gcd}(\alpha g^d, g^{d-1}(d\alpha g'+\alpha' g),r)\ge N_{g^{d-1}}(0,r),
           \end{equation}
 since $\alpha$ and $g$ are entire functions.
 Consequently, 
 $$N_{g}(0,r)\le_{\exc} \varepsilon' \max_{1\le j\le n} \{T_{u_j}(r)\}.$$
 
 \end{proof}
 
 \section{Proofs of  Theorem \ref{main_thm_1}}\label{thm1}
We first show the following.
\begin{theorem}\label{formain_thm1}
	Let $u_0,u_1,\dots,u_n$ be  nonconstant entire functions without zeros, i.e. $\mathbf{u}=(u_{0},\ldots,u_{n}):\mathbb C\to \mathbb G_m^{n+1}$. Let $K$ be a subfield of  the small field $K_\mathbf{u}$ w.r.t. $\mathbf{u}$ such that $u_j'/u_j\in K$, $ 0 \le j\le n$, and $a'\in K$ for any $a\in K$.  Let $d\ge 2$ be an integer.   Let $G$ be a nonconstant  homogeneous polynomials in $K[x_0,\hdots,x_n]$ with no repeated nonmonomial factors in $K[x_0,\hdots,x_n]$. 

If $u_0,\dots,u_n$ are algebraically   independent over $K$, then for any $\epsilon >0$	
\begin{align}\label{truncate_forthm_1}
 N_{G(\mathbf{u})}(0,r)-N^{(1)}_{G(\mathbf{u})}(0,r)\le \epsilon T_{\mathbf{u}}(r),
\end{align}
where $\mathbf{u}:=(u_0,u_1,\hdots,u_n)$.
	
If we assume furthermore that 
there is a $z_0\in\mathbb C$   such that all the coefficients of $G$ are holomorphic  at $z_0$  and  none of  the functions $G(1,0,\dots, 0),\dots,$ $G(0,\dots,0,1)$  vanishes at $z_0$, then
	\begin{align*}
	   N^{(1)}_{G(\mathbf{u})}(0,r)\ge_{\rm exc}  (\deg  G-\epsilon)\cdot T_{\mathbf{u}}(r). 
	\end{align*}
	\end{theorem}
\begin{proof} 
Let $z_0\in\mathbb C$.  If $v_{z_0}( G(\mathbf{u}))\ge 2$, then it follows from \eqref{fuvalue} that $v_{z_0}( D_{\mathbf{u}}(G)(\mathbf{u}))= v_{z_0}( G(\mathbf{u}))-1$.
Hence, 
$$
\min\{v_{z_0}^+(G(\mathbf{u})),v_{z_0}^+(D_{\mathbf{u}}(G)(\mathbf{u}))\}\ge v_{z_0}^+( G(\mathbf{u}))-\min\{1,v_{z_0}^+( G(\mathbf{u}))\}.
$$
Consequently,
\begin{align}\label{truncate}
N_{\gcd}(G(\mathbf{u}), D_{\mathbf{u}}(G)(\mathbf{u}),r)\ge N_{G(\mathbf{u})}(0,r)-N^{(1)}_{G(\mathbf{u})}(0,r).
\end{align}
By Lemma \ref{coprime}, $G$ and $D_{\mathbf{u}}(G)$ are comprime and hence we can apply Theorem \ref{gcd_moving_unit} to get 
\begin{align}\label{gcdUPB}
N_{\gcd}(G(\mathbf{u}), D_{\mathbf{u}}(G)(\mathbf{u}),r)\le_{\rm exc}  \frac{\epsilon}{2} T_{\mathbf{u}}(r)
\end{align}
for any $\epsilon>0$.  Then we have
\begin{align}\label{truncate7}
 N_{G(\mathbf{u})}(0,r)-N^{(1)}_{G(\mathbf{u})}(0,r) \le_{\exc}  \frac{\epsilon}{2}  T_{\mathbf{u}}(r).
\end{align}
This proves the first assertion.

On the other hand, the assumption   that $u_0,\dots,u_n$ are algebraically   independent over $K$ together with  the existence of $z_0$   such that all the coefficients of $G$ are holomorphic  at $z_0$  and   none of  the $G(1,0,\dots, 0),\dots, G(0,\dots,0,1)$  vanishes at $z_0$  allows us to use Theorem \ref{SMTmoving} with polynomials $ G $, $x_0,\hdots,x_n$ to obtain 
\begin{align}\label{applySMT}
 N_{G(\mathbf{u})}(0,r)\ge_{\rm exc} \deg G(1- \frac{\epsilon}{2\deg G} )  T_{\mathbf{u}}(r)=(\deg G- \frac{\epsilon}{2})  T_{\mathbf{u}}(r),
\end{align}
as  $u_i$ is entire without zero for every  $0\le i\le n$. Combining \eqref{applySMT} with \eqref{truncate7}, we obtain the second assertion.
\end{proof}

\begin{proof}[of Theorem \ref{main_thm_1}]
By Lemma \ref{u'/u}, $u_j'/u_j\in  K_{\mathbf{u}} $, $1\le j\le n$, and  by Lemma \ref{a'}  the condition  $a'\in K_{\mathbf{u}}$ for any $a\in  K_{\mathbf{u}}$ also holds.  Therefore,  the first assertion follows from applying both Proposition \ref{Borel1co} and Theorem \ref{formain_thm1} with   $K=K_{\mathbf{u}}$. To deduce the second assertion from Theorem \ref{formain_thm1} applied with   $K=K_{\mathbf{u}}$, we only have to  note that  
$G(1,0,\dots, 0),\dots, G(0,\dots,0,1)$  and the coefficients of $G$ are meromorphic functions   and thus  there is a $z_0\in \mathbb C$  such that all the coefficients of $G$ are holomorphic  at $z_0$  and   none of the $G(1,0,\dots, 0),\dots, G(0,\dots,0,1)$ vanishes at $z_0$ since we assume that  none of them is identically zero.   Finally,  we note that the assumption that $u_0,\dots,u_n$ are entire functions with no zeros of finite order implies that $u_j'/u_j\in \CC(z)$, $1\le j\le n$, and  that $a'\in \CC(z)$ for any $a\in  \CC(z)$; also,  by Proposition \ref{Borel3}, the assumption that  $u_0,\dots,u_n$ are multiplicatively    independent modulo  $\mathbb C$ implies that $u_0,\dots,u_n$  multiplicatively    independent modulo  $\mathbb C(z)$;  Therefore  the third assertion follows from applying  previous arguments with $K=\CC(z)$.   
 \end{proof}

 \section{Proofs of  Theorem  \ref{GG_conj} }\label{thm2}
\begin{proof}[of Theorem \ref{GG_conj}]
Letting $K\subset K_{\mathbf{f}}$ be a subfield, we consider the following arguments.

Let $z_{0}\in\mathbb{C}$   such that  all the coefficients of  all $F_i$, $1\le i\le n+1$ are  holomorphic at $z_0$ and the zero locus of  $F_i$,  $1\le i\le n+1$, evaluating at $z_0$  intersect transversally.  We note these conditions imply that $z_0$ is not a common zero of  the coefficients of $F_i$, for each $1\le i\le n+1$.

For each polynomial $G=\sum_{I}a_{I}{{\bf x}^{I}}\in K[x_{0},\hdots,x_{n}]$, where
$I=(i_{0},\hdots,i_{n})\in\mathbb{Z}_{\ge0}^{n+1}$ and ${\bf x}^{I}=x_{0}^{i_{0}}\cdots x_{n}^{i_{n}}$, we denote by $G({z_{0}}):=\sum_{I}a_{I}(z_{0}){{\bf x}^{I}}$ if all the coefficients of $G$ are holomorphic at $z_0$ and do not vanish simultaneously at  $z_0$.
Denote
by $D_{i}\subset\PP^{n}(\overline{K})$ the divisor (over
$K$) defined by $F_{i}$, by $D_{i}(z_{0})\subset\PP^{n}(\mathbb{C})$
the divisor defined by $F_{i}(z_{0})$, and let $D\coloneqq D_{1}+\cdots+D_{n+1}$.
Since $D_{i}(z_{0})$, $1\le i\le n+1$, intersect transversally,
they are in general position; thus the set of polynomials  $F_{i}$,
$1\le i\le n+1$, is in  weakly general position. 
Then Proposition \ref{HilbertN} implies that the
only $(x_{0},\ldots,x_{n})\in\overline{K}^{n+1}$ with
$F_{i}(x_{0},\ldots,x_{n})=0$ for each $1\le i\le n+1$ is $(0,\ldots,0)$.
Thus the association \textcolor{black}{$P\mapsto[F_{1}^{a_{1}}(P):\hdots:F_{n+1}^{a_{n+1}}(P)]$,
where $a_{i}\coloneqq{\rm lcm}(\deg F_{1},\hdots,\deg F_{n+1})/\deg F_{i}$,
defines a morphism} $\pi:\PP^{n}(\overline{K})\to\PP^{n}(\overline{K})$
over $K$.

It is well-known that the ramification divisor of $\pi$ is  the zero locus of the determinant $J\in K[x_{0},\hdots,x_{n}]$  
of the Jacobian matrix 
$$
\big(\frac{\partial F_i^{a_i}}{\partial x_j}\big)_{1\le i\le n+1, 0\le j\le n}
$$ 
of $\pi$.
Our plan is to show that there exists an irreducible factor $\tilde G$ of $J$ in $K_{\mathbf{f}}[x_{0},\hdots,x_{n}]$ such that the corresponding moving hypersurfaces of $\tilde G$, $F_{1},\ldots,F_{n+1}$ are in weakly general position.
Furthermore, we will show that  $\tilde G(\mathbf{f})$ has very few zeros and hence conclude that $\mathbf{f}$ is algebraically degenerate by applying Theorem \ref{SMTmoving}  for $\tilde G$, $F_{1},\ldots,F_{n+1}$.

For this purpose, we look at the specialization of $J$ at $z_0$.  
Denote by 
$\pi|_{z_{0}}=[F_{1}^{a_{1}}(z_{0}) :\hdots:F_{n+1}^{a_{n+1}}(z_{0}) ]:\PP^{n}(\mathbb{C})\rightarrow\PP^{n}(\mathbb{C})$, which is a morphism since $F_1(z_0),\hdots,F_{n+1}(z_0)$ are in general position.
Then $J(z_{0})\in\mathbb{C}[x_{0},\hdots,x_{n}]$
is the determinant of the Jacobian matrix of $\pi|_{z_{0}}$. 
Observing that $J$ has a factor $G$ which  denotes  the determinant of  
\begin{align*}
M:= \big(\frac{\partial F_i }{\partial x_j}\big)_{1\le i\le n+1, 0\le j\le n}
 \end{align*} 
  in ${\color{blue}K }[x_{0},\hdots,x_{n}]$, we see that $G(z_0)$ is a factor of $J(z_0)$.   We note that  $G(z_{0})$ is not a constant 
since each $F_i$ is homogeneous and reduced and $\sum_{i=1}^{n+1}F_i\ge n+2$.   We claim that $[G(z_0)=0]$ (the zero locus of $G(z_0)$), $D_1(z_0),\hdots,D_{n+1}(z_0)$ are  is in general position (in $\mathbb P^n(\CC)$). To prove this, it suffices to show that $G(z_0)$ does not vanish at any intersection point of any $n$ divisors among $D_1(z_0),\hdots,D_{n+1}(z_0)$.  By rearranging the index, it suffices to consider that $P\in \cap_{i=1}^n D_i(z_0)$ and show that $G(z_0)(P)\ne0$.  Since the $D_i(z_0)$'s are in general position,  we see that $F_{n+1}(z_0)(P)\ne 0$. 
Using the Euler formula 
$$
\sum_{j=1}^{n}\frac{\partial F_i(z_0)}{\partial x_j}x_j=\deg F_i(z_0) \cdot F_i(z_0), 
$$
we obtain 
\begin{equation*}
\begin{split}
        x_0G(z_0) =\det \begin{pmatrix}
        d_1 F_1(z_0)& \frac{\partial F_1(z_0)}{\partial x_1} & \cdots & \frac{\partial F_1(z_0)}{\partial x_n} \\
        \vdots & \vdots & \vdots & \vdots\\
        d_{n+1}F_{n+1}(z_0) & \frac{\partial F_{n+1}(z_0)}{\partial x_1} & \cdots & \frac{\partial F_{n+1}(z_0)}{\partial x_n} 
    \end{pmatrix},
\end{split}
\end{equation*}
and hence
$$
 x_0(P)(G(z_0))(P)=(-1)^{n+1} d_{n+1}(F_{n+1}(z_0))(P)\det \left( \frac{\partial F_i(z_0)}{\partial x_j}(P) \right)_{1\le i, j\le n},
$$
where $d_i:=\deg F_i=\deg (F_i(z_0))$.
Since $D_1(z_0),\hdots,D_{n+1}(z_0)$ intersect transversally, we see that $\det \left( \frac{\partial F_i(z_0)}{\partial x_j}(P) \right)_{1\le i, j\le n}\ne 0$.  Then $G(z_0)(P)\ne 0$ as  $F_{n+1}(z_0)(P)\ne 0$.  This proves our claim; it actually shows that every irreducible factor of  $G(z_0)$ is in general position with  $F_i(z_0)$, $1\le i\le n+1$. Hence, there is a nonconstant irreducible factor $\tilde G$ of $G$ (and hence of $J$)  in $K[x_{0},\hdots,x_{n}]$ such that $\tilde G$,  $F_{1},\ldots,F_{n+1}$
is in weakly general position.
 
We note that the finite map
$\pi:\PP^{n}(\overline{K})\to\PP^{n}(\overline{K})$
is defined over $K$; by the definition of $D$, it induces a finite morphism  $\tilde \pi:=[F_{1}^{a_{1}}/F_{n+1}^{a_{n+1}}:\hdots:F_{n}^{a_{n}}/F_{n+1}^{a_{n+1}}]: \PP^{n}\setminus {\rm Supp}(D)\mapsto\mathbb{G}_{m}^{n}$, which is a morphism between affine varieties over  $K$. Denote by $Y$ the zero locus of $\tilde G$ in $\PP^{n}\setminus {\rm Supp}(D)$; by our construction, $Y$ is contained in the ramification divisor of $\tilde\pi$.
Then there exists an irreducible polynomial $A\in K[y_1,\hdots,y_n]$ such that the vanishing order of $\tilde\pi ^*A$ along $Y$ is at least 2.  Let $\tilde A$ be the homogenization of $A$.  Then this construction gives
$\pi^*\circ \tilde A=\tilde G^2H$ for some $H\in  K[x_{0},\hdots,x_{n}]$. 
Now let $\mathbf{f}=(f_{0},\hdots,f_{n}):\CC\to\PP^{n}$
be a holomorphic map, where $f_{0},\hdots,f_{n}$ are entire functions
without common zeros, such that $\mathbf{u}\coloneqq\pi(\mathbf{f})=(F_{1}(\mathbf{f})^{a_{1}},\ldots,F_{n+1}(\mathbf{f})^{a_{n+1}})$
is a tuple of entire functions without zeros. From the equality $\tilde A(\mathbf{u})=(\pi^*\circ \tilde A)(\mathbf{f})=\tilde G^2(\mathbf{f})H(\mathbf{f})$, it follows that for each $z\in\mathbb{C}$ with $v_{z}(\tilde G(\mathbf{f}))>0$, i.e. $\mathbf{f}(z)\in Y$,
we have 
\begin{align}\label{poleH}
v_{z}(\tilde A(\mathbf{u}))\ge 2v_{z}(\tilde G(\mathbf{f}))+\min\{0,v_{z}( H(\mathbf{f}))\}\ge v_{z}(\tilde G(\mathbf{f}))+1+\min\{0,v_{z}( H(\mathbf{f}))\}.
\end{align}
Since  $f_{0},\hdots,f_{n}$ are entire functions, the nonnegative number $-\min\{0,v_{z}( H(\mathbf{f}))\}$ is bounded by the number of poles of the coefficients of $H$ at $z$.   Since  $N_{\beta}(\infty,r)\le T_{\beta}(r)+O(1)={\rm o}(  T_{\mathbf{f}}(r))$ for any $\beta\in K$, it follows from \eqref{poleH}  that  
\[
N_{\tilde G(\mathbf{f})}(0,r)\le N_{A(\mathbf{u})}(0,r)-N_{A(\mathbf{u})}^{(1)}(0,r)+{\rm o}(  T_{\mathbf{f}}(r)).
\]

 Now consider in the cases where $K=K_{\mathbf{f}}$ and where $K=\CC(z)$. As noticed in the proof of Theorem \ref{GG_conj},  the property required for $(K,\mathbf{u})$ in Theorem \ref{formain_thm1} is satisfied. By  \eqref{truncate_forthm_1} in Theorem \ref{formain_thm1}  and that
$T_{\mathbf{u}}(r)=O(T_{\mathbf{f}}(r))$, we have 
\begin{align}
N_{\tilde G(\mathbf{f})}(0,r)\le_{\rm exc}\epsilon T_{\mathbf{f}}(r)\label{RamifiedG}
\end{align}
 for every $\epsilon>0$.
 
Finally, since the set of polynomials $\tilde G$,  $F_{1},\ldots,F_{n+1}$
is in weakly general position, Theorem \ref{SMTmoving}
shows that if  $\mathbf{f}$ is algebraically nondegenerate over $K$,
then  we reach the contradiction that for every $\epsilon>0$  
\begin{align*}
\left(1-\epsilon\right)T_{\mathbf{f}}(r)\le\frac{1}{\deg \tilde G}N_{\tilde G(\mathbf{f})}(0,r)+\sum_{j=1}^{n+1}\frac{1}{\deg F_{j}}N_{F_{j}(\mathbf{f})}(0,r)\le_{\rm exc}\epsilon T_{\mathbf{f}}(r),
\end{align*}
where the second inequality uses \eqref{RamifiedG} and the hypothesis
that each entire function $F_{i}(\mathbf{f})$ has no zeros. 
 This shows
that $\mathbf{f}$ is algebraically degenerate over $K$,  and proves the first part of the desired conclusions when the case where $K=K_{\mathbf{f}}$ is considered.

It remains to consider when $F_i\in\mathbb C[x_{0},\hdots,x_{n}]$ and $\mathbf f$ is of finite order.  Our previous argument for the case where $K=\mathbb C(z)$ shows that  $\mathbf{f}$ is algebraically degenerate over $\mathbb C(z)$.  By considering the transcendence degree of the extension generated by $F_{1}(\mathbf{f}),\ldots,F_{n+1}(\mathbf{f})$ over $\CC(z)$, we see that they are algebraic dependent over $\mathbb C(z)$.  Since they are units of finite order, Proposition \ref{Borel3} implies that $F_{1}(\mathbf{f}),\ldots,F_{n+1}(\mathbf{f})$ are multiplicatively  dependent modulo  $\mathbb C$.  As the coefficients of all $F_i$ are in $\mathbb C$, it implies that $\mathbf{f}$ is algebraically degenerate over $\mathbb C$. 
 \end{proof}

\section{Proof of Theorem \ref{mainthm3}}\label{thm3}
  The proof of Theorem \ref{mainthm3} is an adaption and generalization of the proof of   \cite[Theorem 3]{corvaja2013algebraic} to the complex situation.  We reformulate and generalize the arguments in  \cite{corvaja2013algebraic} into the  following two lemmas to complete the proof.

\begin{lemma}
\label{mainlemma} Let $d$ and $q$ be positive integers. Let $\mu_{1},\hdots,\mu_{n}\in\mathbb{C}^{*}$
and $\mathbf{u}=(u_{1},\hdots,u_{n})=(e^{\mu_{1}z^{q_{1}}},\hdots,e^{\mu_{n}z^{q_{n}}})$
with $ q=q_1= q_2=\cdots=q_{\ell}> q_{\ell+1}\ge\cdots\ge q_n  \ge 1$ for some $1\le\ell\le n$. Assume that $u_{1},\hdots,u_{n}$
are  algebraic  independent  over $\CC(z)$; this means that $\CC(z)[u_{1}^{\pm 1},\hdots,u_{n}^{\pm 1}]$ can be regarded as the ring of Laurent polynomials in variables $u_{1},\hdots,u_{n}$ over $\CC(z)$. Let $A_{i}\in \mathbb{C}(z)[u_{1}^{\pm 1},\hdots,u_{n}^{\pm 1}]$,
$0\le i\le d-1$. Assume that the polynomial  $F(Y):=Y^{d}+A_{d-1}Y^{d-1}+\hdots+A_{1}Y+A_{0}\in \mathbb{C}(z)[u_{1}^{\pm 1},\hdots,u_{n}^{\pm 1}][Y]$
is irreducible  and  its discriminant $\Delta\in\mathbb{C}(z)[u_{1}^{\pm 1},\hdots,u_{n}^{\pm 1}]$ is square-free.
Suppose that there exists an entire function $g$ such that $F(g)=0$.
 Then there exists  $Q\in\mathbb{C}(z)][u_{\ell+1}^{\pm 1},\dots,u_n^{\pm 1}]$
such that 
\begin{align*}
\Delta=Qu_{1}^{m_{1}}\cdots u_{\ell}^{m_{\ell}},
\end{align*}
where $m_i\in\mathbb Z$ for $1\le i\le \ell$.
\end{lemma}

 The proof of Lemma \ref{mainlemma} will be given at the end of this section.
 
\begin{lemma}
\label{separate} Let $k$ be a field of characteristic zero.  Suppose that  $f(Y)\in  k[U_{1}^{\pm 1},\hdots,U_{n}^{\pm 1}][Y]$ is a monic and irreducible polynomial in $Y$ such that its discriminant $\Delta\in k[U_{1}^{\pm 1},\hdots,U_{n}^{\pm 1}]$ (with respect to $Y$) euqals $QU_{1}^{m}$ for some  $Q\in  k[U_{2}^{\pm 1},\hdots,U_{n}^{\pm 1}] $ and $m\in\mathbb{Z}$. Then we have
\begin{align*}
f(U_{1},\hdots,U_{n},Y)=U_{1}^{s}P(U_{2},\hdots,U_{n},U_{1}^{t}Y+A(U_{1},\hdots,U_{n})),
\end{align*}
where $P\in k[U_{2}^{\pm1},\hdots,U_{n}^{\pm1},W]$, $A(U_{1},\hdots,U_{n})\in k[U_{1}^{\pm1},\hdots,U_{n}^{\pm1}]$ 
and    $s,t\in\mathbb{Z}$. 
\end{lemma}

We note that Lemma \ref{separate} is verified in the middle of the proof of \cite[Theorem 3]{corvaja2013algebraic} for $n=2$  (See \cite[Eq. (12)]{corvaja2013algebraic}.) based on the following claim:
{\it Let  $\mathbf K$ be an algebraically closed field of characteristic 0 and let $F\in  \mathbf K[V,Y]$ be an irreducible polynomial monic and of degree $e$ in $Y$ such that its discriminant with respect to $Y$ is a constant times a power of $V$.   Then $g(V,Y)=(Y-a(V))^e-bV^s$ where $a\in \mathbf K[V]$, $b\in \mathbf K^*$ and $s$ is an integer prime to $e$. 
It is clear that their arguments extend naturally to all positive integer $n$ by taking $\mathbf K$ to be an algebraic closure of $k$ if $n=1$; and an algebraic closure of $k( U_{2},\hdots,U_{n})$ if $n\ge 2$.
}
We will not repeat the proof here.

\begin{proof}[of Theorem \ref{mainthm3}]
 We prove the desired conclusion by induction on $n$. In the case where $n=0$, i.e. each $A_i$ is in $\CC(z)$, the assumption $F(g)=0$ says that the entire function $g$ is algebraic over $\CC(z)$ and hence  $T_{g}(r)=O(\log r)$.
By Theorem \ref{rational},  we conclude that $g\in\CC(z)$ and thus $g\in\CC[z]\subset\mathcal E_0$. This settles down the base case of the induction.

Suppose that $n>0$.  We may rearrange the $u_i$   such that
$$
 q= q_1=\cdots=q_{i_1}>q_{i_1+1}=\cdots=q_{i_2}>\cdots >q_{i_b+1}=\hdots=q_n.
$$
 As $u_{1},\hdots,u_{n}$ are multiplicatively
independent modulo $\mathbb C$, it follows from Proposition   \ref{Borel3} that they are
algebraically independent over $\CC(z)$. Note that $\mathbb{C}(z)[u_{1}^{\pm 1},\hdots,u_{n}^{\pm 1}] \subset \mathcal K_q$ and thus our hypothesis implies that $F(Y)$ is irreducible over $\mathbb{C}(z)[u_{1}^{\pm 1},\hdots,u_{n}^{\pm 1}]$.
By Lemma \ref{mainlemma},  there exists  $Q\in\mathbb{C}(z)[u_{i_1+1}^{\pm 1},\dots,u_n^{\pm 1}]$
such that 
\begin{align*}
\Delta=Qu_{1}^{m_{1}}\cdots u_{i_1}^{m_{i_1}},
\end{align*}
where $m_i\in\mathbb Z$ for $1\le i\le  i_1$.  By Lemma \ref{separate}, we have

\begin{align}\label{discriminantDe}
F(Y)=u_{1}^{s_1}P( u_{1}^{t_1} Y-c),
\end{align}
for some $s_1\in\mathbb Z$,  $P\in \CC(z)[u_{2}^{\pm1},\hdots,u_{n}^{\pm1}][W]$, $t_1\in\mathbb Z$ and $c\in \CC(z)[u_{1}^{\pm1},\hdots,u_{n}^{\pm1}]$.
From \eqref{discriminantDe}, we see that the leading coefficient of $P(W)$ in $W$ must be a power of $u_1$ since $F(Y)\in \mathbb{C}(z)[u_{1}^{\pm 1},\hdots,u_{n}^{\pm 1}][Y]$ is monic in $Y$. But since the coefficients of $P$ are in $\CC(z)[u_{2}^{\pm1},\hdots,u_{n}^{\pm1}]$, it follows that $P$ is monic in $W$. Similarly, we see that  the discriminant of $P$ with respect to $W$ is $\Delta_{P}=Qu_{2}^{m_{2}}\cdots u_{i_1}^{m_{i_1}}$, which is square-free.
 Note also that $P$ is  irreducible in $W$ over $\mathbb{C}(z)[u_{2}^{\pm 1},\hdots,u_{n}^{\pm 1}]$ since $F(Y)$ is irreducible in $Y$. Now \eqref{discriminantDe} gives  $P(u_{1}^{t_1} g-c)=u_{1}^{-s_1} F(g)=0$, which implies that $u_{1}^{t_1} g-c\in\mathcal E_q$ by induction hypothesis; since $c\in\mathcal E_q$ and $u_1\in \mathcal E_q^{*}$, we conclude $g\in \mathcal E_q$ as desired. 
 \end{proof}

The strategy of  proving  Lemma \ref{mainlemma} is motivated by the first part of the proof of   \cite[Theorem 3]{corvaja2013algebraic}.
We  will first show that thenumber of zeros of  $F'(g)$ is ``small".  Then 
we can use Theorem \ref{trunborel}, a truncated second main theorem,
 to the algebraic relation of $F'(g)$ and $u_1,\hdots,u_n$ to 
conclude the assertion on that the discriminant $\Delta$.  

\begin{proof}[of Lemma \ref{mainlemma}]
The assertion is clear when $d=1$. Therefore, we assume that $d\ge2$.
 Denote by $\alpha_{1},\hdots,\alpha_{d}$ all the roots of $F$ in a fixed algebraic closure of $\CC(z,u_{1},\hdots,u_{n})$ such that $\alpha_1=g$ is an entire function. Then we have $\Delta=\prod_{i<j}(\alpha_{i}-\alpha_{j})^{2}$, which is not zero since $F$ is irreducible.

Let $H:=F(Y)/(Y-g)=\prod_{1<i\le d}(Y-\alpha_{i})$.
The division algorithm shows that $H\in \CC(z)[u_{1}^{\pm1},\hdots,u_{n}^{\pm1},g][Y]$.
Letting $G\in \CC(z)[u_{1}^{\pm1},\hdots,u_{n}^{\pm1},g]$ be the discriminant of $H$ with respect to $Y$, we note that the entire function 
$G$ is equal to $\prod_{1<i<j\le d}(\alpha_{i}-\alpha_{j})^{2}$, which is nonzero.

Note that $F'(g)=\prod_{i>1}(g-\alpha_{i})$ is a nonzero entire function, where $F'$ denotes the formal derivative of $F$ with respect to $Y$. Hence
\begin{align}
\Delta=F'(g)^2G.\label{disc1}
\end{align}
Let $\epsilon_1$ be a sufficiently small positive number. We claim
that 
\begin{align}
N_{F'(g)}{\color{blue}(0,} r)\le_{{\rm exc}}\epsilon_1\cdot r^{q}+o(r^{q}).\label{counting}
\end{align}
By our assumption   that $\Delta\in\CC(z)[u_{1}^{\pm1},\hdots,u_{n}^{\pm1}]$ is square-free, we may write 
\begin{align}
\Delta=u_{1}^{m_1}\cdots u_{n}^{m_n}A(u_{1},\hdots,u_{n})\label{disc2}
\end{align} 
with integers $m_1,\hdots,m_n$ and a square-free polynomial $A\in \CC(z)[x_{1},\hdots,x_{n}]$. 
Since $u_{i} $, $1\le i\le n$, are entire functions without zeros, we have from \eqref{disc1} and \eqref{disc2}
that there exists a non-zero entire function $\gamma$ such that 
\begin{align}
A(u_{1},\hdots,u_{n})=F'(g)^{2}\gamma.\label{discriminant2}
\end{align}
We note that elements in $ \CC(z)$ have slow growth
with respect to $\mathbf{u}$. As $u_{1},\hdots,u_{n}$ are
algebraically independent over $ \CC(z)$, we have
\eqref{counting} from Lemma \ref{dth_power_count}.

   Since $F(g)=0$ with $F\in\CC(z,u_{1},\hdots,u_{n})[Y]$, it follows that $g$ is algebraic over $\CC(z,u_{1},\hdots,u_{n})$ and that $F'\in\CC(z,u_{1},\hdots,u_{n})[Y]$, so $F'(g)$ is also algebraic over $\CC(z,u_{1},\hdots,u_{n})$. Hence there exists a  polynomial  $R\in\CC[z][U_{1},\hdots,U_{n},Y]$ of the fewest terms such that  
\begin{align}
R(u_{1},\hdots,u_{n},F'(g))=0.\label{algebraicRelation}
\end{align}
Since $F'(g)\ne0$, we may assume that with respect to the variable $Y$, the  polynomial $R$ has a nonzero constant term. Note that \eqref{algebraicRelation} can be written as
\begin{align}
\sum_{I=(i_{0},\hdots,i_{n})\in\Sigma}c_{I}F'(g)^{i_{0}}u_{1}^{i_{1}}\cdots u_{n}^{i_{n}}=0,\label{algebraicRelation2}
\end{align}
where $\Sigma$ is the set of vectors $(i_{0},\hdots,i_{n})$ corresponding
to monomials appearing in $R$ and $c_{I}\in \CC[z]$, not identically zero.  By construction, no proper subsum of the left hand side of \eqref{algebraicRelation2} vanishes, and $\Sigma$ must contain a vector $I_0=(0,a_1,\hdots,a_n)$ starting with $0$. 

Note that in the left hand side of \eqref{algebraicRelation2}, the nonzero term corresponding to $I_0$ has only finitely many zeros. Thus there is some $h\in\CC[z]$ such that all terms in the left hand side of the following equality equivalent to \eqref{algebraicRelation2}
\begin{align}
\sum_{I=(i_{0},\hdots,i_{n})\in\Sigma}h^{-1}c_{I}F'(g)^{i_{0}}u_{1}^{i_{1}}\cdots u_{n}^{i_{n}}=0\label{algebraicRelation2red}
\end{align}
are entire functions without common zeros. 

Let $\epsilon$ be a sufficiently small positive real number.  
Applying Theorem \ref{trunborel} to \eqref{algebraicRelation2red},  with the estimate in  \eqref{counting} and Lemma \ref{gineq},
we have 
\begin{align} 
T_{F'(g)^{i_{0}}u_{1}^{i_{1}-a_{1}}\cdots u_{n}^{i_{n}-a_{n}}}(r)\le_{{\rm exc}}\epsilon\cdot r^{q}+{\rm o}(r^{q})\label{small}
\end{align}
for each $(i_{0},\hdots,i_{n})\in\Sigma$. 

On the other hand, the assumption that $u_{1},\hdots,u_{n}$
are multiplicatively independent modulo $\mathbb C$ implies that if the index vector $ J= (j_1,\hdots,j_{\ell})\ne (0,\hdots,0)$, then 
\begin{align}\label{big} 
T_{u_{1}^{j_{1}}\cdots u_{n}^{j_{n}}}(r)=\gamma_{J}r^{q}+  o(r^q)  
\end{align}
for some positive real $\gamma_{J}$. 
 We first notice that \eqref{small} implies that  if  $(a_1,\cdots,a_{\ell})\ne (0,\hdots,0)$, then the first $\ell+1$ coordinates of any other vector in $\Sigma$ with $i_0=0$ must be equal to $(0,a_1,\cdots,a_{\ell})$.
Let $M_1=U_1^{a_1}\cdots U_{\ell}^{a_{\ell}}$ and $\tilde R=M_1^{-1}R. $   We have $
{\tilde R}(u_{1},\hdots,u_{n},F'(g))
=   \sum_{I=(i_{0},\hdots,i_{n})\in\Sigma'}c_{I}F'(g)^{i_{0}}u_{1}^{i_{1}}\cdots u_{n}^{i_{n}}$, where $\Sigma':= 
\{ (i_{0},i_1-a_1,\hdots,i_{\ell}-a_{\ell},i_{\ell+1}\hdots,i_{n}): (i_{0},\hdots,i_{n})\in\Sigma\}$. Note  
that if  $(0,i_1\hdots,i_{n})\in\Sigma'$, then $i_1=\cdots=i_{\ell}=0$.
  Since ${\tilde R}(u_{1},\hdots,u_{n},F'(g))=0$, it follows for any $(i_0,i_1,\hdots,i_n), (j_0,j_1,\hdots,j_n)\in\Sigma'$ that
\begin{align*}
 &T_{u_{1}^{i_{1}j_0-i_0j_1}\cdots u_{n}^{i_{n}j_0-i_0j_n}}(r)\\
\le & j_0 T_{F'(g)^{i_{0}}u_{1}^{i_{1}}\cdots u_{n}^{i_{n}}} (r)+ i_0 T_{F'(g)^{j_{0}}u_{1}^{j_{1}}\cdots u_{n}^{j_{n}}}(r)\\
\le & \epsilon r^{q}+{\rm o}(r^{q}),
\end{align*}
where the second inequality is obtained as \eqref{small} is;
then \eqref{big} implies that $j_0(i_0,i_1,\cdots,i_{\ell})=i_0(j_0,j_1,\cdots,j_{\ell})$.
Then there is some tuple $(c_0,\hdots,c_{\ell})$ of integers with $c_0\ge0$ such that for every $I=(i_0,i_1,\hdots,i_n)\in\Sigma'$ we have that $(i_0,i_1,\cdots,i_{\ell})=e_I(c_0,\hdots,c_{\ell})$ for some nonnegative integer $e_I$.  Hence we see that 
\begin{align}\label{R1}
0=\tilde R(u_{1},\hdots,u_{n},F'(g))=R_1(u_1^{c_1}\hdots u_{\ell}^{c_{\ell}}F'(g)^{c_0})\end{align}
with some
$R_1\in\CC[z,u_{\ell+1},\dots,u_{n}][X]$. 
 
It  follows immediately from \eqref{R1} that $u_1^{c_1}\hdots u_{\ell}^{c_{\ell}}F'(g)^{c_0}$ is algebraic  over $\CC(z)(u_{\ell+1},\hdots,u_{n})$. Letting $u:=u_1^{c_1}\hdots u_{\ell}^{c_{\ell}}$ and recalling that $\alpha_2,\hdots,\alpha_n$ are all the conjugates of $g=\alpha_1$     over $\CC(z)(u_1,\hdots,u_n)$, we see that $uF'(\alpha_{i})^{c_0}$ is algebraic  over $\CC(z)(u_{\ell+1},\hdots,u_{n})$ for all $i=1,\hdots,d$.
Since $\Delta=\prod_{i=1}^d F'(\alpha_{i})$, we deduce that $u^d\Delta^{c_0}$ is algebraic over $\CC(z)(u_{\ell+1},\dots,u_{n})$. For $j=1,\hdots,\ell$, by considering the minimal polynomial of $u^d\Delta^{c_0}$ over $\CC(z)(u_{\ell+1},\dots,u_{n})$, we find that the differentiation of $u^d\Delta^{c_0}$ with respect to the transcendental element $u_j$ over $\CC(z)$ is zero. Hence we conclude that $u^d\Delta^{c_0}\in\mathbb{C}[z][u_{\ell+1}^{\pm 1},\dots,u_n^{\pm 1}]$ and thus the desired conclusion follows. 
\end{proof}


\begin{thebibliography}{99}

\bibitem{CT}
{ L.~Cspuano and A.~Turchet},
Lang-Vojta Conjecture over function fields for surfaces dominating $\mathbb G_m^2$,
arXiv:1911.07562.

\bibitem{CGPastenLogic}
{D.~Chompitaki, N.~Garcia-Fritz, H.~Pasten, T.~Pheidas and X.~Vidaux},
The Diophatine problem for rings of exponential polynomials, arXiv:2004.00612 




\bibitem{corvaja2013algebraic} 
{P.~Corvaja and U.~Zannier}, Algebraic hyperbolicity of ramified covers of $\GG_m^2$ (and integrap points on affine subsets of $\PP^2$), \textsc{\emph{J. Differential Geom.}}\textsc{ }\textbf{\textsc{93}}\textsc{ (2013), no.~3 355--377.}
  

\bibitem{DethloffTan2011}
{G.~Dethloff and T.~V.~Tan},
  \emph{A second main theorem for moving hypersurface targets}, 
  \emph{Houston J. Math.} \textbf{37},
  (2011), 79--111
  
   
\bibitem{green1974functional}
{M.~Green}, On the functional equation
  {$f^{2}=e^{2\phi_{1}}+e^{2\phi_{2}}+e^{2\phi_{3}}\ $} and a new {P}icard
  theorem,  \emph{Trans. Amer. Math. Soc.} \textbf{195} (1974), 223--230.

\bibitem{green1975some}
{M.~Green}, Some {P}icard theorems for holomorphic
  maps to algebraic varieties,  \emph{Amer. J. Math.} \textbf{97} (1975),
  43--75.

\bibitem{guo2019quotient}
{J.~Guo}, The quotient problem for entire functions,
  \emph{Canad. Math. Bull.} \textbf{62} (2019),  no.~3, 479--489.

\bibitem{guo2019quotient2}
{J.~Guo}, Quotient problem for entire functions with moving targets,  \emph{targets, Houston J. Math.}, \textbf{46} (2020),  no.~3, 627--650.

\bibitem{GSW20}
{J.~Guo, C.~L Sun and  J.~T.-Y. Wang}, On the d-th roots of exponential polynomials and related problems arising from Green-Griffiths-Lang conjecture,  J. of Geometric Analysis, to appear. 

\bibitem{heittokangas2018zero}
{J.~Heittokangas,
  K.~Ishizaki, K.~Tohge,
  and Z.-T. Wen}, Zero distribution and division
  results for exponential polynomials,  \emph{Israel J. Math.} \textbf{227}
  (2018),  no.~1, 397--421.

\bibitem{lang1987introduction}
{S.~Lang}, \emph{Introduction to complex hyperbolic
  spaces}, Springer-Verlag, New York, 1987.

\bibitem{levin2019greatest}
{A.~Levin and J.~T.-Y. Wang},
  Greatest common divisors of analytic functions and nevanlinna theory on
  algebraic tori,  \emph{J. Reine Angew. Math.}, \textbf{767}
  (2020), 77--107.





\bibitem{noguchi2007degeneracy}
{J.~Noguchi, J.~Winkelmann,
  and K.~Yamanoi}, Degeneracy of holomorphic curves
  into algebraic varieties,  \emph{J. Math. Pures Appl. (9)} \textbf{88}
  (2007),  no.~3, 293--306.
  
  \bibitem{noguchi2008semiabelian}
{J.~Noguchi, J.~Winkelmann,
  and K.~Yamanoi}, The second main theorem for holomorphic curves into semi-abelian varieties II,  \emph{Forum Math.  } \textbf{20}
  (2007),   469--503.

\bibitem{ritt1929algebraic}
{J.~F. Ritt}, Algebraic combinations of exponentials,
  \emph{Trans. Amer. Math. Soc.} \textbf{31} (1929),  no.~4, 654--679.


\bibitem{ru2001nevanlinna}
{M.~Ru}, \emph{Nevanlinna theory and its relation to
  {D}iophantine approximation}, World Scientific Publishing Co., Inc., River
  Edge, NJ, 2001.

\bibitem{ru2004truncated}
{M.~Ru and J.~T.-Y. Wang},
  Truncated second main theorem with moving targets,  \emph{Trans. Amer. Math.
 Soc.} \textbf{356} (2004),  no.~2, 557--571.


  
  \bibitem{Vo}
{P. Vojta}, \emph{Diophantine approximations and value distribution theory},
  Lecture Notes in Mathematics, vol. 1239, Springer-Verlag, Berlin, 1987.
  
  \bibitem{Voj89}\textsc{P.~Vojta,} \emph{A refinement of Schmidt's subspace theorem}, Amer. J. Math.  \textbf{111}
(1989), no. 3, 489--518.
  



\bibitem{vojta2009diophantine}
{P.~Vojta}, Diophantine approximation and {N}evanlinna
  theory,  in \emph{Arithmetic geometry}, \emph{Lecture Notes in Math.}
  \textbf{2009}, Springer, Berlin, 2011, 111--224.
  

\bibitem{waerden1967}
{B.~L.~van~der~Waerden}, \emph{Moderne {A}lgebra}, vol. 2, 5th ed., Springer-Verlag, Berlin, 1967; English transl., Ungar, New York, 1970. 
  
  
 \bibitem{Yamanoi2004}
{ K.~Yamanoi},   
  The second main theorem for small functions and related problems. \emph{Acta Math.} \textbf{192} (2004), no. 2, 225--294

\bibitem{zannier2000proof}
{U.~Zannier}, A proof of {P}isot's {$d$}th root
  conjecture,  \emph{Ann. of Math. (2)} \textbf{151} (2000),  no.~1, 375--383.

\end{thebibliography}
\end{document}